\documentclass[12pt]{article}
\usepackage[utf8]{inputenc}
\usepackage{tikz}
\usepackage[nottoc]{tocbibind}
\usepackage[colorinlistoftodos]{todonotes}
\usepackage{hyperref}
\usepackage[]{geometry}
\usepackage{amsmath,amsfonts,amssymb,amsthm,booktabs}
\usepackage{caption}
\usepackage{subcaption}

\renewcommand{\P}{\mathbb{P}}
\newcommand{\E}{\mathbb{E}}
\newcommand{\Var}{\mathrm{Var}}

\newcommand{\lijn}{\hspace{0.05cm}\big|\hspace{0.05cm}}

\newcommand{\dx}{\text{d}x}

\newcommand{\atot}{A_{tot}}
\newcommand{\la}{\lambda_A}
\newcommand{\lb}{\lambda_B}

\newtheorem{definition}{Definition}
\newtheorem{theorem}{Theorem}

\newtheorem{lemma}{Lemma}

\usepackage{url}

\newcommand{\footremember}[2]{%
   \footnote{#2}
    \newcounter{#1}
    \setcounter{#1}{\value{footnote}}%
}
\newcommand{\footrecall}[1]{%
    \footnotemark[\value{#1}]%
} 

\begin{document}


\title{Degree distributions in AB random geometric graphs}
\author{
Clara Stegehuis\footremember{ut}{Department of Electrical Engineering, Mathematics and Computer Science, University of Twente, the Netherlands} $\&$ Lotte Weedage\footrecall{ut}}

\maketitle
\begin{abstract}
In this paper, we provide degree distributions for AB random geometric graphs, in which points of type $A$ connect to the closest $k$ points of type $B$. The motivating example to derive such degree distributions is in 5G wireless networks with multi-connectivity, where users connect to their closest $k$ base stations. In this setting, it is important to know how many users a particular base station serves, which gives the \textit{degree} of that base station. To obtain these degree distributions, we investigate the distribution of area sizes of the $k-$th order Voronoi cells of $B$-points. Assuming that the $A$-points are Poisson distributed, we investigate the amount of users connected to a certain $B$-point, which is equal to the degree of this point. In the simple case where the $B$-points are placed in an hexagonal grid, we show that all $k$-th order Voronoi areas are equal and thus all degrees follow a Poisson distribution. However, this observation does not hold for Poisson distributed $B$-points, for which we show that the degree distribution follows a compound Poisson-Erlang distribution in the 1-dimensional case. We then approximate the degree distribution in the 2-dimensional case with a compound Poisson-Gamma degree distribution and show that this one-parameter fit performs well for different values of $k$. Moreover, we show that for increasing $k$, these degree distributions become more concentrated around the mean. This means that $k$-connected $AB$ random graphs balance the loads of $B$-type nodes more evenly as $k$ increases. Finally, we provide a case study on real data of base stations. We show that with little shadowing in the distances between users and base stations, the Poisson distribution does not capture the degree distribution of these data, especially for $k>1$. However, under strong shadowing, our degree approximations perform quite good even for these non-Poissonian location data.\\

\end{abstract}




\section{Introduction}
Spatial point processes have many applications, ranging from the distribution of stars in the Milky Way~\cite{babu_spatial_1996} to the dispersal of biological species~\cite{othmer_models_1988,renner_equivalence_2013}. One application of spatial processes which has received significant interest is in wireless networks~\cite{elsawy_stochastic_2013}. In the typical setting, a wireless network consists of base stations and users that are distributed according to some spatial process, and users connect to the nearest base station. \\

In 5G networks, the new concept multi-connectivity is introduced. In multi-connected networks, users connect to the $k>1$ nearest base stations. Having multiple connections can make the internet connection faster and more reliable~\cite{tesema2015mobility}. In multi-connectivity, the size distribution of $k-$th order areas is a quantity of importance. Indeed, from the are size distribution, one can obtain the distribution of the number of users that connect to a particular base station. This quantity is necessary to derive analytical expressions for important network statistics such as the network capacity and outage probabilities.\\

In this paper we therefore investigate the degree distribution in $AB$ random geometric graphs \cite{Penrose1999OnGraph}, a random graph model for multi-connected networks in which points of type $A$ connect to the $k$ closest points of type $B$. We are interested in the size of the area in which a given $B$-point is the $k$-th closest to a given $A$-point to derive the degree distribution of $B$-points. Here we assume that $A$-points are distributed as a Poisson point process. Beside applications in multi-connected networks, other applications are in the $G_{n,k}$ random graph, where $n$ points connect to their closest $k$ neighbors. For these types of random graphs, only high-level characteristics are known, such as the parameters such that the resulting graph is connected~\cite{balister_connectivity_nodate}. Results on the area in which a given point is $k$-th closest would enable to derive the degree distribution of these random graphs as well, which could give more insights into the behavior of these graphs under dynamic processes such as epidemics or cascading failures. Other applications of such areas are in $k$-nearest neighbor classification~\cite{sitawarin_adversarial_2020} or in plant ecology~\cite{magnussen_gamma-poisson_nodate}.\\

A property of Poisson processes is that a specified area size provides the distribution of the amount of users in that area. Thus, rather than analyzing the degree distribution directly, in this paper, we first derive expressions for the size distributions of the areas in which $B$ points are $k$-th closest to $A$ points.\\

This degree distribution depends on the spatial distribution of the $B$-points, as different spatial distributions give different areas in which $B$-points are $k$-th closest. We focus on two popular location models: a hexagonal grid model and a Poisson point process. The Poisson point process is one of the most popular spatial processes, due to its mathematical properties that make it relatively easy to analyze. Examples of spatial processes that are often modeled by Poisson processes are wireless networks~\cite{elsawy_stochastic_2013}, the dispersal of biological species~\cite{othmer_models_1988} or in forestry~\cite{stoyan2000recent}. 
The Poisson point process has proven useful to obtain several quantities of interest analytically. For example, in single-connected wireless networks, the Poisson process allows to derive the probability of network outages or the capacity that each user in the network receives~\cite{elsawy_stochastic_2013}.  
The hexagonal grid is a simpler spatial process model, which has been used in many applications such as modeling wireless networks~\cite{nasri2015tractable}, ecology~\cite{birch2007rectangular} and agent based modeling~\cite{brown2005spatial}. The area sizes in this hexagonal grid are easier to analyse in terms of Voronoi area sizes, but other quantities of interest may be more difficult to derive, as the locations of points in a hexagonal grid are dependent, in contrast to the Poisson point process. In wireless networks, the hexagonal grid model often overestimates network performance measures such as outage probabilities compared to real data, while the Poisson point process slightly underestimates them \cite{lee2013stochastic}.\\

For both location models, we start with investigating the area sizes in the 1-dimensional case, after which we will show the 2-dimensional case.
For $k=1$, the problem reduces to finding the size-distribution of Voronoi cells, cells which indicate in which region a given point is the closest of all points for the grid and the Poisson point Process. Unfortunately, no exact characterisation of the sizes of these Poisson-Voronoi cells exist, although several approximations from numerical simulations exist~\cite{jarai-szabo_size-distribution_2008,weaire1986distribution}.\\

For $k>1$, the area in which a user connects to a given point can no longer be found by standard Voronoi diagrams. For such settings, higher-order Voronoi diagrams or $k$-th order Voronoi diagrams exist~\cite{edelsbrunner_voronoi_1986}. For example, in a second order Voronoi diagram, every cell corresponds to a pair of points $(i_1,i_2)$, such that $i_1$ is the closest, and $i_2$ is the second closest in that area. In a $k$-th order Voronoi diagram, every cell represents the area where a given point is $k$-th closest. These cells are in general nor convex, nor connected, making analysis of $k$-th order Voronoi cells complex~\cite{edelsbrunner_voronoi_1986}. Several results on fast algorithms to construct higher-order Voronoi diagrams exist~\cite{aurenhammer_simple_2011,boissonnat_semidynamic_1993, zavershynskyi_sweepline_2013}, as well as results on the complexity of its cells~\cite{bohler_complexity_2015} as well as the cell shapes~\cite{edelsbrunner_voronoi_1986, martinez-legaz_structure_2019}. However, to our knowledge, no results on the cell sizes of higher-order Voronoi cells or $k$-th order Voronoi cells exist under any type of underlying point process. \\

In this paper, we investigate the sizes of $k$-th order Voronoi area's for $k>1$ in order to find the degree distributions for the Poisson process and the hexagonal grid. In the 1-dimensional setting, we obtain an exact result of the distributions of the regions where given points are the $k$-th closest. Interestingly, these regions are \emph{equal} in distribution for all $k$. We also derive an exact expression of the degree distribution in the 1-dimensional setting. We show that the degree distribution becomes more concentrated when $k$ grows. Thus, increasing $k$ balances the load in terms of connections more evenly among the $B$ points.  \\

In the 2-dimensional setting, we provide exact results for the areas and degree distributions under the hexagonal grid model. Unfortunately, for the Poisson point process no exact results on the Voronoi-area sizes exist even for $k=1$. We therefore turn to numerical simulations instead. We provide one-parameter fitted distributions similar to the well-accepted approximation for $k=1$ to approximate the distribution of the areas where a given point is $k$-th closest in Poisson-Voronoi cells. With these parameter fits, we find a compound Poisson-Gamma degree distribution. Moreover, we show that the 1-dimensional Poisson case, for which we found an exact degree distribution, also approximates the 2-dimensional degree distribution well, especially when $k$ becomes large.  In both the Poisson point process and the hexagonal grid, we show that the coefficient of variation of the degree distributions decrease when $k$ increases, which means that the load of the network gets more evenly balanced among the $B$-points.\\

Finally, we investigate a case study of real data of base stations in the Netherlands. Interestingly, while these base stations are not distributed according to a Poisson point process, we show that when some randomness is present in the random graph connections, in the form of shadowing present, the degree distribution of these non-Poissonian data is well approximated by our results for the one- or 2-dimensional Poisson case for $k > 1$. When this randomness is not present, the fits for the degree distribution that is obtained from the Poisson Point processes still fits reasonably well for $k=1$, but the fit significantly deteriorates for larger $k$. This also indicates that it is possible that for some data a Poisson point process is a suitable model for $k$-connected AB-graphs when $k=1$, but not for larger values of $k$.  \\ 

In Section \ref{sec:degreedistr}, we derive analytical results for the degree distribution for the hexagonal grid in 1 and 2 dimensions and for the 1-dimensional Poisson setting. Moreover, we provide an approximate one-parameter fit for the 2-dimensional setting. We then show the quality of this fit,  and derive the degree distribution for Poisson distributed points. Then, in Section \ref{sec:limiting_degree}, we investigate a case study on non-Poissonian real data of base station locations. We show that while these real base station locations are not distributed as a Poisson process, under high shadowing, our approximations still work quite well to predict the degree distributions of these base stations. Furthermore, we show that without shadowing, the degree distribution is still predicted quite well for $k=1$, but for higher values of $k$, it becomes worse. This indicates that while the frequently made Poisson point process assumption may be justified for $k=1$, the correlations between the presence of different points in non-Poissonian data can make the Poisson assumption unjustified for higher-order connectivity levels.

\section{Degree distribution in $k$-connectivity}\label{sec:degreedistr}
We model $k$-connectivity in an $AB$ random geometric graph \cite{penrose2014continuum} consisting of points of type $A$ and points of type $B$. Type $A$ points are distributed by a homogeneous Poisson Point Process with density $\la$, while type $B$-points have a general spatial distribution. Every point in $A$ connects to the nearest $k$ points in $B$. An example of 2-connectivity is given in Figure \ref{tikz:example2con}.\\

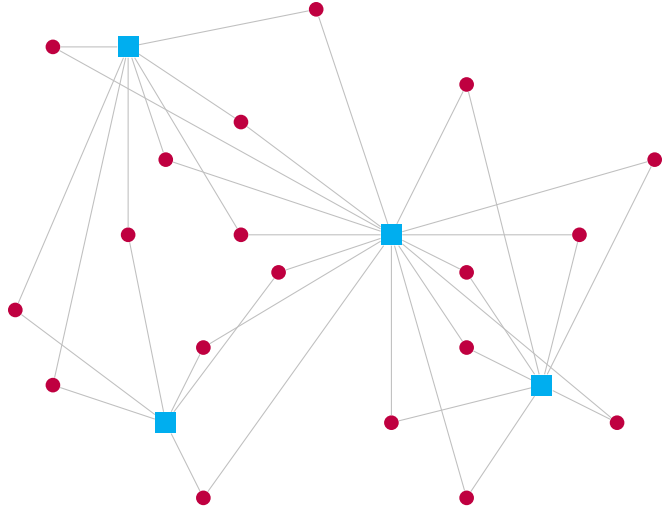
\begin{figure}[tbp]
    \centering
\begin{tikzpicture}
        \node[shape = rectangle, fill, color = cyan] (BS1) at (1,0.5) {};
        \node[shape = rectangle, fill, color = cyan] (BS2) at (-2.5,3) {};
        \node[shape = rectangle, fill, color = cyan] (BS3) at (3,-1.5) {};
        \node[shape = rectangle, fill, color = cyan] (BS4) at (-2,-2) {};
        
        \node[shape = circle, fill, color = purple, scale =0.5] (U1) at (-2,1.5) {};
        \node[shape = circle, fill, color = purple, scale =0.5]  (U2) at (-1,2) {};
        \node[shape = circle, fill, color = purple, scale =0.5] (U3) at (2,2.5) {};
        \node[shape = circle, fill, color = purple, scale =0.5]  (U4) at (0,3.5) {};
        \node[shape = circle, fill, color = purple, scale =0.5] (U5) at (-2.5,0.5) {};
        \node[shape = circle, fill, color = purple, scale =0.5] (U6) at (-1.5,-1) {};
        \node[shape = circle, fill, color = purple, scale =0.5] (U7) at (2,-3) {};
        \node[shape = circle, fill, color = purple, scale =0.5]  (U8) at (-1.5,-3) {};
        \node[shape = circle, fill, color = purple, scale =0.5] (U9) at (2,-1) {};
        \node[shape = circle, fill, color = purple, scale =0.5] (U10) at (4,-2) {};
        \node[shape = circle, fill, color = purple, scale =0.5] (U11) at (2,0) {};
        \node[shape = circle, fill, color = purple, scale =0.5] (U12) at (3.5,0.5) {};
        \node[shape = circle, fill, color = purple, scale =0.5] (U13) at (-3.5,-1.5) {};
        \node[shape = circle, fill, color = purple, scale =0.5] (U14) at (-4,-0.5) {};
        \node[shape = circle, fill, color = purple, scale =0.5] (U15) at (4.5,1.5) {};
        \node[shape = circle, fill, color = purple, scale =0.5] (U16) at (-3.5,3) {};
        \node[shape = circle, fill, color = purple, scale =0.5] (U17) at (-1,0.5) {};
        \node[shape = circle, fill, color = purple, scale =0.5] (U18) at (-0.5,0) {};
        \node[shape = circle, fill, color = purple, scale =0.5] (U19) at (1,-2) {};
        
        \draw[color = lightgray] (U16) -- (BS2);
        \draw[color = lightgray] (U1) -- (BS2);
        \draw[color = lightgray] (U2) -- (BS2);
        \draw[color = lightgray] (U4) -- (BS2);
        \draw[color = lightgray] (U5) -- (BS2);
        
        \draw[color = lightgray] (U14) -- (BS4);
        \draw[color = lightgray] (U13) -- (BS4);
        \draw[color = lightgray] (U8) -- (BS4);
        \draw[color = lightgray] (U6) -- (BS4);
        
        \draw[color = lightgray] (U18) -- (BS1);
        \draw[color = lightgray] (U17) -- (BS1);
        \draw[color = lightgray] (U11) -- (BS1);
        \draw[color = lightgray] (U3) -- (BS1);
        
        \draw[color = lightgray] (U15) -- (BS3);
        \draw[color = lightgray] (U12) -- (BS3);
        \draw[color = lightgray] (U9) -- (BS3);
        \draw[color = lightgray] (U7) -- (BS3);
        \draw[color = lightgray] (U10) -- (BS3);
        \draw[color = lightgray] (U19) -- (BS3);
        
        \draw [color = lightgray](U13) -- (BS2);
        \draw [color = lightgray](U14) -- (BS2);
        \draw [color = lightgray](U17) -- (BS2);
        
        \draw [color = lightgray](U5) -- (BS4);
        \draw [color = lightgray](U18) -- (BS4);
        \draw [color = lightgray](U19) -- (BS1);
        
        \draw [color = lightgray](U1) -- (BS1);
        \draw [color = lightgray](U16) -- (BS1);
        \draw [color = lightgray](U4) -- (BS1);
        \draw [color = lightgray](U2) -- (BS1);
        \draw [color = lightgray](U9) -- (BS1);
        \draw [color = lightgray](U6) -- (BS1);
        \draw [color = lightgray](U8) -- (BS1);
        \draw [color = lightgray](U12) -- (BS1);
        \draw [color = lightgray](U15) -- (BS1);
        \draw [color = lightgray](U10) -- (BS1);
        \draw [color = lightgray](U7) -- (BS1);
        
        \draw [color = lightgray](U3) -- (BS3);
        \draw [color = lightgray](U11) -- (BS3);
    \end{tikzpicture}
    \caption{Example of a 2-connected $AB$ random graph, where circles are $A$-points and squares are $B$-points}
    \label{tikz:example2con}
\end{figure}

We are interested in the degree distribution of the $B$-points, which we obtain by using the fact that every $A$-point in the degree-$j$-Voronoi cell of a random point in $B$ connects to this $B$-point if $j \leq k$. Thus, if the area of a certain cell is known, so is the distribution of the number of $A$-points ($N_A$) in this area:
\begin{align}
    \P(N_A(x) = n) &= \frac{(\la x)^n}{n!}e^{-\la x},\label{eq:number_of_points}
\end{align}
where $x$ denotes the size of this area.

In order to find the degree distribution of a random point in $B$, denoted by $D_B$, we then simply need to find the sum of the sizes of the degree-$j$-Voronoi cells for $1 \leq j \leq k$, which we call $X_{\leq k}$:
\begin{align}
\mathbb{P}(D_{B}=n) &=\int_{0}^{\infty} \mathbb{P}\left(N_A\left(x\right) = n \lijn X_{\leq k}=x\right)  f_{X_{\leq k}}(x) \dx , \label{eq:degree_distribution}
\end{align}
for area distribution $f_{X_{\leq k}}(x)$.\\

In the following sections, we derive expressions for the degree distribution of $B$-points for two spatial distributions of the $B$-points. In Section \ref{sec:hexagonal}, we assume that the points in $B$ are placed in a hexagonal grid, and then in Section \ref{sec:poisson}, we investigate the setting in which they are distributed as a Poisson point process. 

We now provide some definitions that will be used throughout this paper:
\begin{definition}[$k$-th order area]\label{def:areas}
The area on which point $i \in B$ is the $k-$th point if sorted by increasing distance, is denoted by $X_k(i)$.
\end{definition}

\begin{definition}[$\leq k$-th order area]\label{def:k-tharea}
The area on which point $i \in B$ is the $j-$th point where $1 \leq j \leq k$ if sorted by increasing distance, is denoted by $X_{\leq k}(i)$.
\end{definition}

\subsection{Hexagonal grid}\label{sec:hexagonal}
We first investigate the degree distribution and thus area sizes in the regular grid. For one dimension, this means that points $B$ are placed on a line with equal spacing, and for two dimensions we investigate the hexagonal grid.

\subsubsection{1-dimensional case}

The following theorem gives an expression for the degree distribution in the 1-dimensional hexagonal grid:

\begin{theorem}\label{th:1dimhexagonal}
    The degree $D_B$ of a randomly chosen $B$-point in the 1-dimensional regular grid is Poisson distributed with parameter $d$,  where $d$ is the distance between two consecutive points.
\end{theorem}

\begin{proof}
    In this 1-dimensional case, the distance between two consecutive points is $D_i = B_i - B_{i-1} = d$ is equal for all $i \in B$, where $B_i$ is the coordinate of point $i \in B$. An example of the 1-dimensional grid is given in Figure \ref{tikz:uniform_distr}. From this picture, it can be seen that the $X_k(i) = X_{k+1}(i)$ for all $i \in B$ and $k$, as every $X_k(i)$ consists of two area's of size $d/2$. Therefore, by \eqref{eq:number_of_points}, the degrees $D_B$ are Poisson distributed with parameter $d$.
\end{proof}

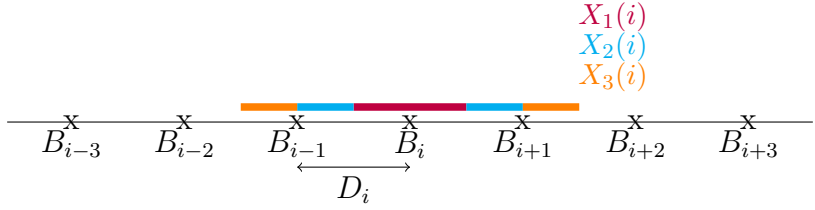
\begin{figure}[tbp]
\centering
    \begin{tikzpicture}
        \node (v1) at (0,-2.5) {};
        \node (v2) at (11,-2.5) {};
        \draw (v1) -- (v2);
        \draw[color = purple, line width = 1 mm] (4.75,-2.3) -- (6.25,-2.3);
        \draw[color = cyan, line width = 1 mm] (6.25,-2.3) -- (7,-2.3);
        \draw[color = cyan, line width = 1 mm] (4,-2.3) -- (4.75,-2.3);
        \draw[color = orange, line width = 1 mm] (7,-2.3) -- (7.75,-2.3);
        \draw[color = orange, line width = 1 mm] (3.25,-2.3) -- (4,-2.3);
        \node at (1,-2.5) {x};
        \node at (2.5,-2.5) {x};
        \node at (4,-2.5) {x};
        \node at (5.5,-2.5) {x};
        \node at (7,-2.5) {x};
        \node at (8.5,-2.5) {x};
        \node at (10,-2.5) {x};
        \node at (1,-2.8) {$B_{i-3}$};
        \node at (2.5,-2.8) {$B_{i-2}$};
        \node at (4,-2.8) {$B_{i-1}$};
        \node at (5.5,-2.8) {$B_i$};
        \node at (7,-2.8) {$B_{i+1}$};
        \node at (8.5,-2.8) {$B_{i+2}$};
        \node at (10,-2.8) {$B_{i+3}$};
        \draw[<->]  (4,-3.1) -- (5.5,-3.1);
        \node at (4.75,-3.4) {$D_i$};
        \node[color = purple] at (8.2,-1.1) {\small$X_1(i)$};
        \node[color = cyan] at (8.2,-1.5) {\small$X_2(i)$};
        \node[color = orange] at (8.2,-1.9) {\small$X_3(i)$};
    \end{tikzpicture}
    \caption{$X(i, k)$ in an 1-dimensional regular lattice}
    \label{tikz:uniform_distr}
\end{figure}

\subsubsection{2-dimensional case}
We now turn to the 2-dimensional setting. The following theorem shows that in the 2-dimensional setting (Figure \ref{fig:triangular_grid}), the degree distribution is again Poisson, but with another parameter:
\begin{figure}[tbp]
    \centering
    \includegraphics[width=.6\textwidth]{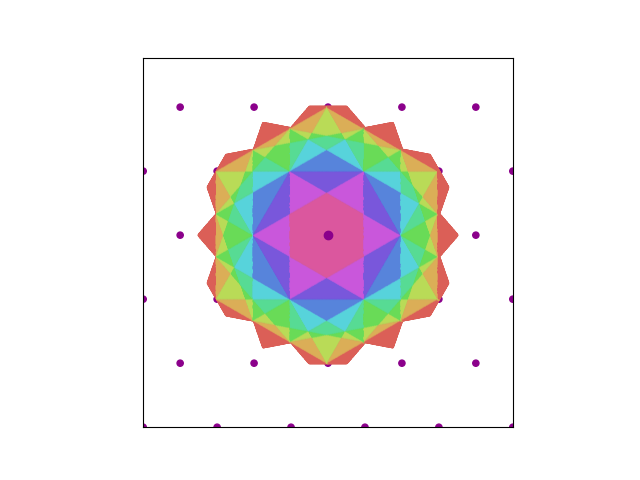}
    \caption{Points and areas in a hexagonal lattice. Every color is an area in which the center point is the $k$-nearest point}
    \label{fig:triangular_grid}
\end{figure}

\begin{theorem}\label{th:2dhexagonal}
    The degree $D_B$ of a randomly chosen $B$-point in the 2-dimensional hexagonal grid is Poisson distributed with parameter $\la k \atot / N$, for $A$-point density $\lambda_A$, $k$ connections, a total area $A_{tot}$ and $N = |B|$.
\end{theorem}

To prove this theorem, we need to know the distribution of the area sizes, for which we introduce the following lemma.

\begin{lemma}[Equal areas]\label{lem:2dhexagonalarea}
In a regular lattice, the $k$-th order area for every grid point $i \in B$ is equal:
\begin{align}
    X_k(i) = X_{k+1}(i), \hspace{1cm} \forall i \in B, \forall k \geq 1
\end{align}
\end{lemma}

\begin{proof}
Let us assume we have a lattice on a torus with $N = |B|$ points and a total area $\atot$. Because we have a regular grid, we know that for all $i$ and $j$,
\begin{align}
    X_1(i) = X_1(j) = \frac{\atot}{N}.
\end{align}
Since all $B$-points are equal and symmetric, it holds that for all $i, j \in B$ and $k$:
\begin{align}
    X_k(i) = X_k(j).
\end{align}
Furthermore, the sum over all $k-$th order areas of the points in the grid sum up to $\atot$, since every area in the grid will be the $k-$th closest area to one of the $B-$points in the grid:
\begin{align}
    \sum_{i} X_k(i) = \atot.
\end{align}

This shows that $X_k(i) = X_k(j) = \atot/N$ for all $i, j \in B$ and $k$.
\end{proof}

Now, we prove Theorem \ref{th:2dhexagonal} using Lemma \ref{lem:2dhexagonalarea}.
\begin{proof}
Since $X(i, s)$ is equal for every $s \geq 1$, the area $X_{\leq}(i, k) = k \cdot X(i, 1) = k \atot / N$, where $\atot$ is the total area of the grid and $N$ is the total number of points in that grid. We can now fill in the area in \eqref{eq:number_of_points}, as we assumed the $A$-points are Poisson distributed. Therefore, the degree distribution of a randomly chosen $B$-point is:
\begin{align}
    \mathbb{P}(D_{B}=n) &= \P\left(N_A\left(\frac{k\atot}{N}\right) = n\right) = \frac{(\la \frac{k\atot}{N})^n}{n!}e^{-\la \frac{k\atot}{N}}.\label{eq:poisson2d}
\end{align}
This means that $D_B \sim \text{Poisson}\left(\la k \atot/N\right)$. 
\end{proof}

We would like to compare the different degree distributions $D_B$ for different values of $k$. We therefore compute the coefficient of variation of $D_B$:
\begin{align}
    \E(D_B^2) &= \sum_{n=0}^\infty n^2 \cdot \P(D_B = n) = \la \frac{k \atot}{N}\left(1 + \la \frac{k \atot}{N}\right),\\
    \Var(D_B) &= \E(D_B^2) - \left(\E(D_B)^2\right) = \la \frac{k \atot}{N},\\
    c_{V} &= \frac{\sqrt{\la \frac{k \atot}{N}}}{\la \frac{k \atot}{N}} = \left(\la \frac{k \atot}{N}\right)^{-\frac{1}{2}}.\label{eq:cv1}
\end{align}
This shows that the coefficient of variation decreases for increasing $k$. Thus, when increasing the connectivity in an $AB$-random graph, the degree distribution of the $B$-points becomes more concentrated.

\subsection{Poisson Point Process}\label{sec:poisson}
Now, let us investigate the case where $B$-points are distributed as a Poisson point process. Again, we first find an expression for the degree distribution in one dimension, where the points are placed on a line according to a Poisson process with parameter $\lb$ (Figure \ref{tikz:poisson1d}). We then focus on the 2-dimensional problem, where points are distributed as a homogeneous Poisson point process with parameter $\lb$ (Figure \ref{fig:pppgrid}).

\subsubsection{1-dimensional case}
The following theorem provides the degree distribution in the 1-dimensional Poisson case:

\begin{figure}[tbp]
\centering
    \begin{tikzpicture}
        \node (v1) at (0,-2.5) {};
        \node (v2) at (11,-2.5) {};
        \draw (v1) -- (v2);
        \draw[color = purple, line width = 1 mm] (4.85,-2.3) -- (6.45,-2.3);
        \draw[color = cyan, line width = 1 mm] (6.45,-2.3) -- (7,-2.3);
        \draw[color = cyan, line width = 1 mm] (3.95,-2.3) -- (4.85,-2.3);
        \draw[color = orange, line width = 1 mm] (7,-2.3) -- (8,-2.3);
        \draw[color = orange, line width = 1 mm] (3.4,-2.3) -- (3.95,-2.3);
        \node at (1.3,-2.5) {x};
        \node at (2.4,-2.5) {x};
        \node at (4.2,-2.5) {x};
        \node at (5.5,-2.5) {x};
        \node at (7.4,-2.5) {x};
        \node at (8.5,-2.5) {x};
        \node at (10.5,-2.5) {x};
        \node at (1.3,-2.8) {$B_{i-3}$};
        \node at (2.4,-2.8) {$B_{i-2}$};
        \node at (4.2,-2.8) {$B_{i-1}$};
        \node at (5.5,-2.8) {$B_i$};
        \node at (7.4,-2.8) {$B_{i+1}$};
        \node at (8.5,-2.8) {$B_{i+2}$};
        \node at (10.5,-2.8) {$B_{i+3}$};
        \draw[<->]  (4.2,-3.1) -- (5.5,-3.1);
        \node at (4.75,-3.4) {$D_i$};
        \node[color = purple] at (8.2,-1.1) {\small$X_1(i)$};
        \node[color = cyan] at (8.2,-1.5) {\small$X_2(i)$};
        \node[color = orange] at (8.2,-1.9) {\small$X_3(i)$};
    \end{tikzpicture}
    \caption{$X_k(i)$ in an 1-dimensional Poisson process}
    \label{tikz:poisson1d}
\end{figure}
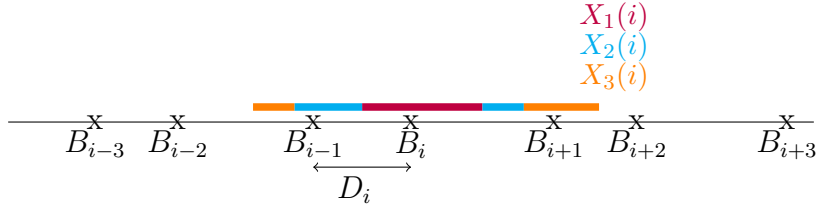

\begin{theorem}\label{th:1dpoisson}
    The degree $D_B$ of a randomly chosen $B$-point in the 1-dimensional Poisson point process has the following distribution function:
    \begin{align}
        \P(D_B = n) &= \frac{\Gamma(2k+n)}{n!\Gamma(2k)} \frac{(2\lb)^{2k} \la^n}{(2\lb + \la)^{2k + n}}.\label{eq:1dpoissonthm}
    \end{align}
We call this distribution the compound Poisson-Erlang distribution.
\end{theorem}
To prove this, we are again interested in the distribution of the area $X_k(i)$ given in Definition \ref{def:areas}.

\begin{lemma}\label{thm:1dpois}
    The $k-$th order area of point $i \in B$, $X_k(i)$ in the 1-dimensional Poisson process is in distribution equal to $\frac{1}{2}D_1 + \frac{1}{2}D_2$, where $D_1$ and $D_2$ are independent and identically exponentially distributed random variables with parameter $\lb$.
\end{lemma}

\begin{proof}

We express $X_k(i)$ in terms of $D_i$:
\begin{align}
    X_k(i) &= \left|\frac{1}{2}(B_i + B_{i-k}) -  \frac{1}{2}(B_{i} + B_{i-k+1}))\right| + \left|\frac{1}{2}(B_{i+k+1} + B_{i}) - \frac{1}{2}(B_{i+k} + B_{i})\right| \nonumber \\
    &= \frac{1}{2}\left( \left(B_{i-k+1}-B_{i-k}) \right) + \left (B_{i+k} - B_{i+k-1}\right)\right) \nonumber \\
    &= \frac{1}{2}\left(D_{i- k + 1} + D_{i+k}\right) \nonumber \\
    &\overset{d}{=} \frac{1}{2}D_1 + \frac{1}{2}D_{2}\label{eq:thm1dpoisson}
\end{align}
where the last step follows from the fact that the random variables $D_i$ are iid by definition.
\end{proof}

As we now know the size distributions of the areas, we can prove Theorem \ref{th:1dpoisson}.
\begin{proof}
Lemma~\ref{thm:1dpois} shows that for every $n \neq m$, $X_n(i) \overset{d}{=} X_m(i)$. As follows from \eqref{eq:thm1dpoisson}, the distribution of $X_k(i)$ is the sum of two exponential distributions with parameter $\lb$, which means that $X_k(i) \sim \text{Erlang}(2, 2\lb)$. By the property of the Erlang distribution, this means that $X_{\leq k}(i) = \sum_{j = 1}^k X_j(i) \sim \text{Erlang}(2k, 2\lb)$. With this distribution for $X_{\leq k}(i)$ we can find the degree distribution, by using \eqref{eq:degree_distribution}:
\begin{align}
    \P(D_B = n) &= \int_0^\infty \P\left(N_A\left(x\right) = n\right) f_{X_{\leq}(i, k)}\left(x\right)\dx \nonumber \\
    &= \frac{\la^n}{n!} \frac{(2 \lb)^{2k}}{\Gamma(2k)}\int_0^\infty x^{n+2k-1} e^{-(2\lb + \la)x}\dx  \nonumber \\
    &= \frac{\Gamma(2k+n)}{n!\Gamma(2k)} \frac{(2\lb)^{2k} \la^n}{(2\lb + \la)^{2k + n}}.\label{eq:1dpoisson}
\end{align}
\end{proof}

 Again, to compare the degree distributions for different values of $k$, we derive the coefficient of variation of $D_B$:
\begin{align}
    \E(D_B) &= k\lambda,\\
    \E(D_B^2) &= \sum_{n=0}^\infty n^2 \P(D_B = n) = k\lambda \left(1 + k\lambda + \frac{1}{2}\lambda\right),\\
    \Var(D_B) &= \E(D_B^2) - \left(\E(D_B)^2\right) = k\lambda + \frac{1}{2} k\lambda^2,\\
    c_{V} &= \frac{\sqrt{k\lambda + \frac{1}{2}k\lambda^2}}{k\lambda} = \sqrt{\frac{1 + \frac{1}{2}\lambda}{k\lambda}}, \label{eq:cv2}
\end{align}
for $\lambda = \la/\lb$. Interestingly, this coefficient of variation again decreases for increasing values of $k$, again with rate $\sqrt{k}$, similar to~\eqref{eq:cv1}. Thus, the degree distribution of $B$-points becomes more concentrated with the same rate in $k$ as for the hexagonal grid.

\subsubsection{2-dimensional case}\label{sec:2d-case}
We now investigate the setting where $B$-points are distributed as a 2-dimensional Poisson point process. In Figure \ref{fig:pppgrid}, we plotted the $k-$th order area's of a random point in the Poisson point process. This figure shows that these areas are not equal and therefore we need to find a different expression for the degree distribution. First, we find an approximation of the area sizes in this setting, and then we show the approximated degree distribution of the 2-dimensional Poisson point process.\\

\begin{figure}[h!]
    \centering
    \includegraphics[width=.6\textwidth]{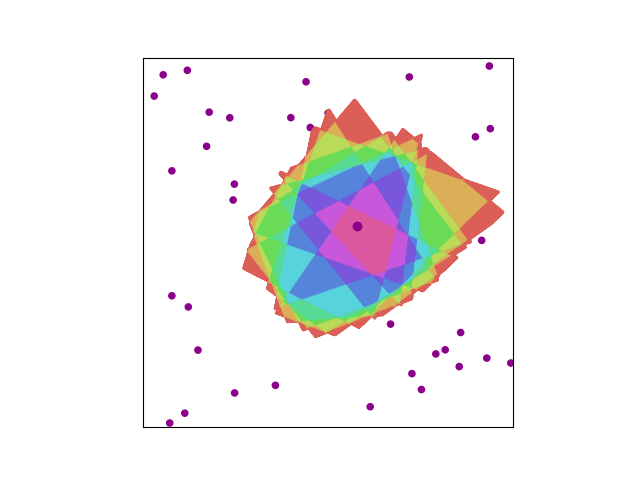}
    \caption{Points and areas in a Poisson point process. Every color is an area in which the center point is the $k$-nearest point.}
    \label{fig:pppgrid}
\end{figure}

\textbf{Fitting the parameters}\\
For the size of the first order area in a Poisson point process, only approximations exist~\cite{dicenzo1989monte,jarai-szabo_size-distribution_2008, weaire1986distribution}. For the $k-$th order areas, we therefore resort to approximations. Equal to \cite{jarai-szabo_size-distribution_2008}, we use a gamma distribution and fit the parameters of this distribution with a simulation, as the authors show that a simple 2-parameter fit gamma distribution is a fair approximation:
\begin{align}
f_{X_{1}}(x)=\frac{3.5^{3.5}}{\Gamma(3.5)} x^{2.5} e^{-3.5 x}. \label{eq:size_distribution}
\end{align}

We extend this parameter fit for higher order Poisson-Voronoi areas.\\


In order to find the degree distribution, we need to find the sum of the sizes of the first to $k$-th order areas, as this is the area on which a $B$-point is the $k$-th or less closest $B$-point, denoted by $X_{\leq k}$. We fit the parameters $a_k$ and $b_k$ in the two-parameter gamma distribution: 
\begin{align}
    f_{X_{\leq k}}(x) &= \frac{{b_k}^{a_k}}{\Gamma({a_k})}x^{{a_k}-1} e^{-{b_k}x}\label{eq:gamma_distribution_paramter_fit}\\
    F_{X_{\leq k}}(x) &= \frac{\gamma({a_k}, {b_k} \cdot x)}{\Gamma({a_k})}\label{eq:gamma_distribution_paramter_fit_cdf},
\end{align}
We assume that the expected area $\E(X_{\leq k})$ is $k$ for every $k$ so that we can simplify \eqref{eq:gamma_distribution_paramter_fit_cdf} with $a_k = k \cdot b_k$. This simplifies the fit to only one parameter:
\begin{align}
    f_{X_{\leq k}}(x) &= \frac{\left(a_k\right)^{a_{k}}}{k^{a_k}\Gamma(a_k)}x^{a_k-1} e^{-\frac{a_k}{k} x} \label{eq:A_leq_k_distribution}\\
    F_{X_{\leq k}}(x) &= \frac{\gamma\left(a_k, \frac{a_k}{k} \cdot x\right)}{\Gamma(a_k)}\label{eq:new_A_leq_k_distribution}
\end{align}

We simulated $n$ Poisson-distributed points on a $\sqrt{n} \times \sqrt{n}$ square and obtained the $k$-th order area of every point in this square by fitting $R$-trees \cite{beckmann1990r} with a precision of $\epsilon$.\\

For the parameter fit, we did $m = 0.7$ million iterations of this algorithm with $n = 100$ points and a precision of $\epsilon^2 = 0.1$, which gives a sample of 7 million points. We used the $\chi^2$-goodness-of-fit-test to find the best parameters 
$a_k = k \cdot b_k$ for the Gamma distribution of $X_{\leq k}$, shown in Table~\ref{tab:simulated_values_area}.\\

Figure \ref{fig:pdf_simulation} shows the goodness of fit of the area distribution. Considering we only used a single-parameter fit, the approximation gives an excellent fit for every $k$. 

\begin{figure}[tbp]
\centering
\begin{minipage}[b]{0.3\textwidth}

\centering
\begin{tabular}{lr}
\toprule
$k$        & $a_k$    \\ \midrule
1          &  $3.53$ \\
2          &   $7.19$    \\
3          &   $11.06$         \\
4          &  $15.21$     \\
5          &   $21.17$ \\ \bottomrule
\end{tabular}
\captionof{table}{Parameters $a_k = k \cdot b_k$ for the size distribution of ${X_{\leq k}}$.}
\label{tab:simulated_values_area}
\end{minipage}
\hfill
\begin{minipage}[t]{0.6\textwidth}
    \centering
    \includegraphics[width = \textwidth]{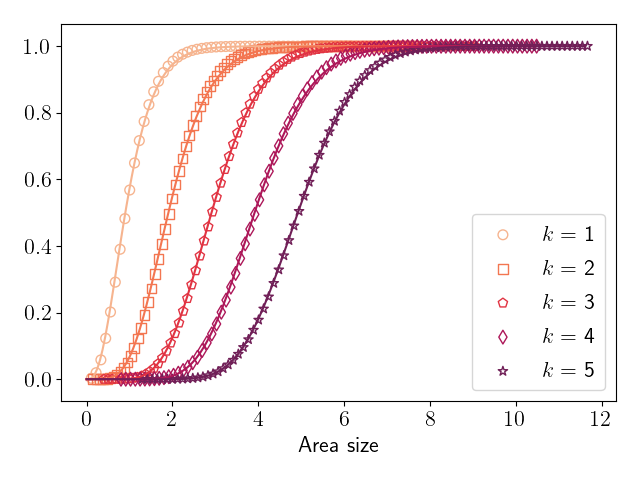}
    \caption{cdf of $X_{\leq k}$, approximated values (line) and observed values (markers).}
    \label{fig:pdf_simulation}
\end{minipage}
\end{figure}

With this area distribution, it is possible to find the degree distribution of $B$ of \eqref{eq:degree_distribution}. Since we assumed that $\E(X_{\leq k}) = 1$ in \eqref{eq:A_leq_k_distribution} while this should be equal to the total area divided by the number of base stations, $\atot / |B| = 1/\lb$, we divide $x$ by $\lb$ in \eqref{eq:degree_distribution} to get the desired expected value and thus the desired distribution:
\begin{align}
    \mathbb{P}(D_{X}=n) &=\int_{0}^{\infty} \mathbb{P}\left(N_Y\left(\frac{x}{\lb}\right) = n \lijn X_{\leq k}=x\right)  f_{X_{\leq k}}(x) \dx \nonumber \\
    &= \frac{\Gamma(n + a_k)}{\Gamma(n+1)\Gamma(a_k)}\frac{a_k^{a_k} (k\lambda)^n}{(k \lambda + a_k)^{a_k + n}},\label{eq:degree_distribution_Poisson}
\end{align}
where $\lambda = \la/\lb$. We call this distribution the compound Poisson-Gamma distribution. The coefficient of variation of $D_B$ is as follows:
\begin{align}
    \E(D_B) &= k\lambda,\\
    \E(D_B^2) &= \sum_{n=0}^\infty n^2 \P(D_B = n) = k\lambda \left(1 + k\lambda + \frac{k \lambda}{a_k}\right),\\
    \Var(D_B) &= \E(D_B^2) - \left(\E(D_B)^2\right) = k\lambda + \frac{(k\lambda)^2}{a_k},\\
    c_{V} &= \frac{\sqrt{k\lambda + \frac{(k\lambda)^2}{a_k}}}{k\lambda} = \sqrt{\frac{1}{k\lambda} + \frac{1}{a_k}}, \label{eq:cv3}
\end{align}
which again decreases for larger values of $k$ as we assume $a_k$ also increases for larger values of $k$ (see Table~\ref{tab:simulated_values_area}). Thus, the degree distribution of $B$ points concentrates at a higher rate in $k$ than the 1-dimensional Poisson process and the hexagonal grid, for which the coefficient of variation decreases as $1/\sqrt{k}$. 

\section{Numerical results on the degree distributions}\label{sec:limiting_degree}
In this section, we compare our analytical results and approximations of the degree distributions of the hexagonal and Poisson grid against simulations. 

\subsection{Regular lattice}
For the hexagonal grid, Theorem \ref{th:2dhexagonal} shows that the degrees are Poisson distributed with parameter $\la k \atot/N$ \eqref{eq:poisson2d}, where $N$ is the number of points and $\atot$ is the total area of the grid. In Figure \ref{fig:limiting_distribution_hexagonal}, we plotted the simulated results together with the analytical Poisson degree distribution. This figure shows that the simulations follow this degree distribution well for all values of $k$.

\begin{figure}[h!]
\begin{subfigure}[!]{0.32\textwidth}
    \centering
    \includegraphics[width = \textwidth]{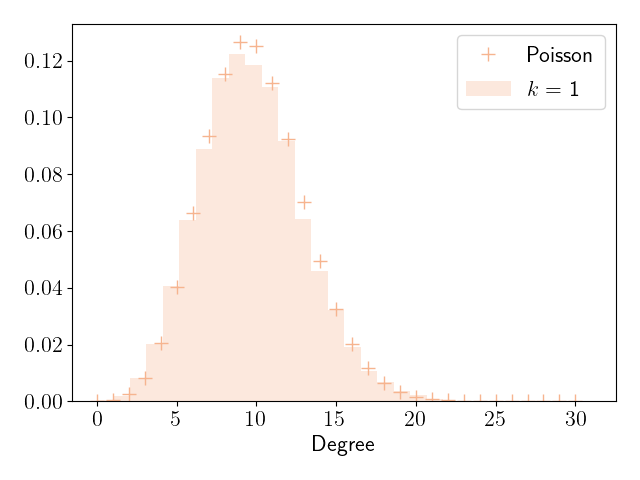}
    \caption{$k = 1$}
    \label{sfig:k1hexagonal}
\end{subfigure}
\hfill
\begin{subfigure}[!]{0.32\textwidth}
    \centering
    \includegraphics[width = \textwidth]{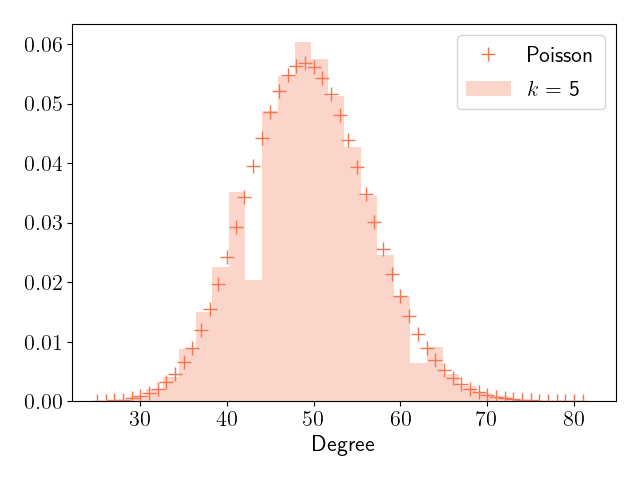}
    \caption{$k = 5$}
    \label{sfig:k5hexagonal}
\end{subfigure}
\begin{subfigure}[!]{0.32\textwidth}
    \centering
    \includegraphics[width = \textwidth]{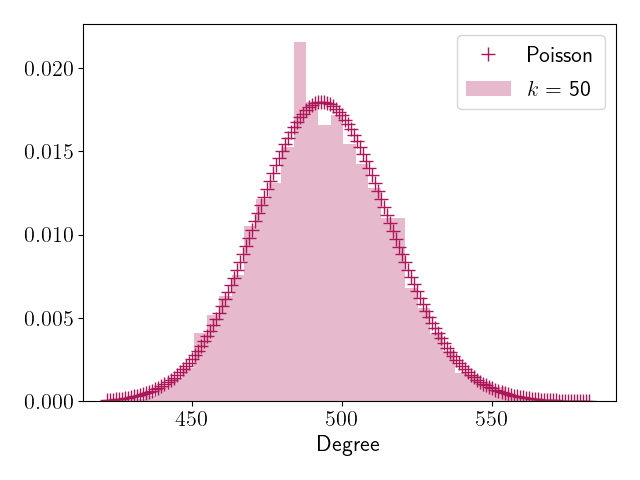}
    \caption{$k = 50$}
    \label{sfig:k50hexagonal}
\end{subfigure}
    \caption{Degree distribution in hexagonal grid for $k = 1, 5$ and $50$ with $22785$ $B-$points and $\la = 0.1$ on a $3000\,m\times3000\,m$ area.}
    \label{fig:limiting_distribution_hexagonal}
\end{figure}

\subsection{Poisson point process}
In Figure \ref{fig:limiting_distribution_Poisson}, we plotted the simulated degree distributions of the 2-dimensional Poisson process for $k = 1, 5$ and $50$ with the one-parameter fit degree distribution for $k = 1$ and $k = 5$ as given in \eqref{eq:degree_distribution_Poisson} and the analytical compound Poisson-Erlang distribution given in Theorem \ref{th:1dpoisson}. This figure shows that the one-parameter fit for $k = 1$ and $k = 5$ fits well. Moreover, the compound Poisson-Erlang degree distribution, which was derived for the 1-dimensional Poisson process, fits reasonably well for the 2-dimensional Poisson process, especially for larger values of $k$. For large values of $k$, the simple 1-dimensional result can also be used instead of the more extensive Gamma distribution with the fitted parameters.

\begin{figure}[h!]
\begin{subfigure}[!]{0.32\textwidth}
    \centering
    \includegraphics[width = \textwidth]{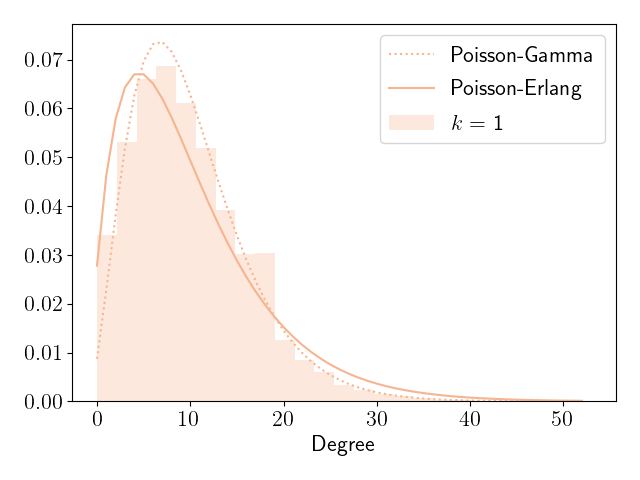}
    \caption{$k = 1$}
    \label{sfig:k1poisson}
\end{subfigure}
\hfill
\begin{subfigure}[!]{0.32\textwidth}
    \centering
    \includegraphics[width = \textwidth]{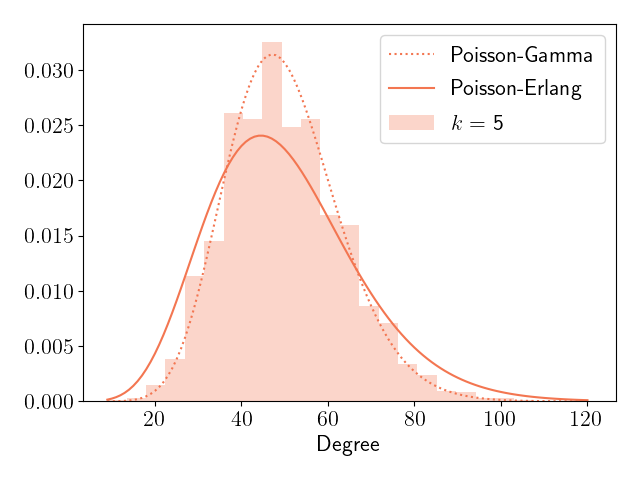}
    \caption{$k = 5$}
    \label{sfig:k5poisson}
\end{subfigure}
\begin{subfigure}[!]{0.32\textwidth}
    \centering
    \includegraphics[width = \textwidth]{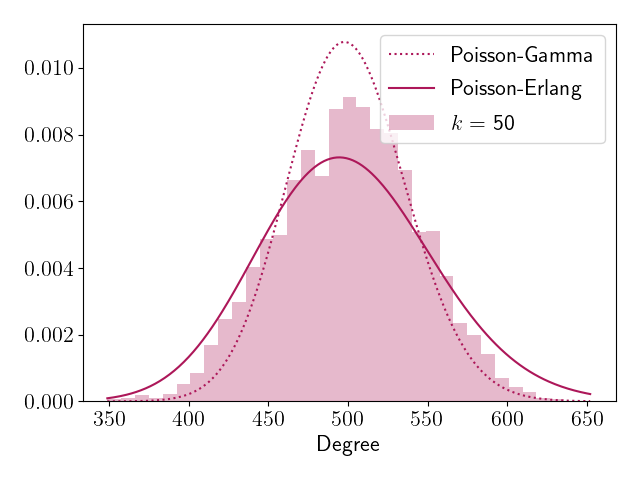}
    \caption{$k = 50$}
    \label{sfig:k50poisson}
\end{subfigure}
    \caption{Degree distribution in Poisson grid for $k = 1, 5$ and $50$ with $\lb = 0.01$ and $\la = 0.1$.}
    \label{fig:limiting_distribution_Poisson}
\end{figure}

\subsection{Real data}
The final case we investigated is the area and degree distribution in a real-world network. We used base station data from OpenCelliD \cite{opencellid} from the Netherlands and focused on the city centre of Enschede (Figure \ref{fig:data_enschede}). While these locations are clearly not distributed as a Poisson point process, we investigate to what extent our approximations for the degree distributions are valid under such non-Poissonian data.\\

\begin{figure}[h!]
    \centering
    \includegraphics[width = 0.5\textwidth]{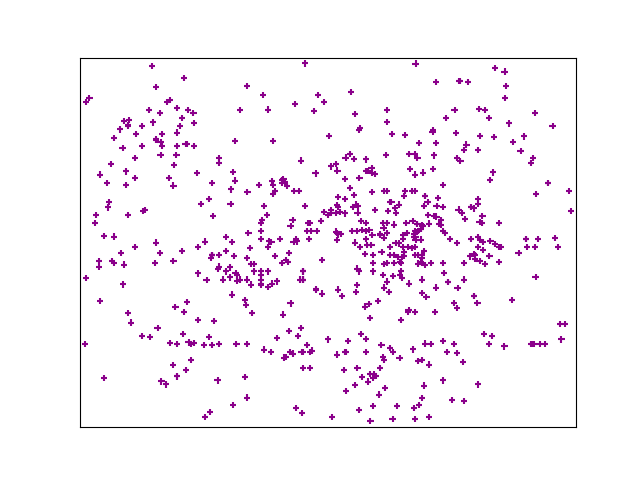}
    \caption{599 base stations in Enschede}
    \label{fig:data_enschede}
\end{figure}

Previous research showed that while base stations are non-Poissonian, investigating the network based on Poisson data can still work under so-called \textit{shadowing} \cite{blaszczyszyn2012quality, ren2011modelling}. Shadowing in the path-loss model can cause perturbations in the observed signal at the user \cite{ren2011modelling}, which causes users to connect to the $k$ base stations with the strongest signal instead of the $k$ closest base stations. Therefore, we incorporate shadowing into this real, non-Poissonian data to investigate the quality of our degree distributions. We use log-normal shadowing \cite{blaszczyszyn2012quality}, which gives a distance after shadowing $d^*(x, y)$ for base station $x$ and user $y$:
\begin{align}
    d^*(x, y) := \frac{d(x, y)}{S_X(y)},
\end{align}
where $d(x,y)$ denotes the real distance between base station $x$ and user $y$ and $S_X(y)$ is a log-normal random variable with mean 1. Then, users connect to the $k$ base stations that have the lowest value of $d^*(x,y)$. In this real data setting, we simulated users with a Poisson process and calculated the degree of every base station. We show results for two different values of the shadowing variance: $\sigma = 0.1$ (weak shadowing) and $\sigma = 1$ (strong shadowing) in Figures \ref{fig:limiting_distribution_real} and \ref{fig:limiting_distribution_real_sigma1}. We plotted the simulations together with the the compound Poisson-Erlang degree distribution \eqref{eq:1dpoisson} and for $k = 1$ and $k = 5$ the fitted compound Poisson-Gamma degree distribution given in \eqref{eq:degree_distribution_Poisson}.\\

\begin{figure}[h!]
\begin{subfigure}[!]{0.32\textwidth}
    \centering
    \includegraphics[width = \textwidth]{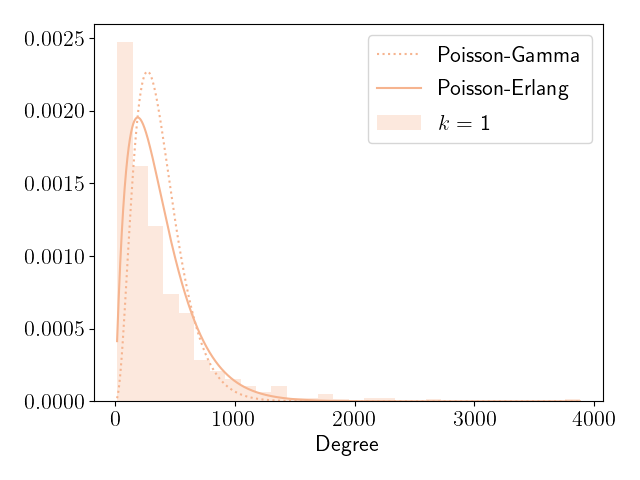}
    \caption{$k = 1$}
    \label{sfig:k1reallow}
\end{subfigure}
\hfill
\begin{subfigure}[!]{0.32\textwidth}
    \centering
    \includegraphics[width = \textwidth]{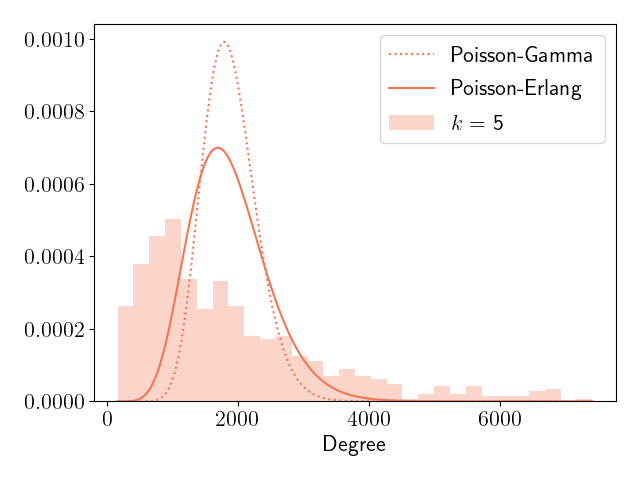}
    \caption{$k = 5$}
    \label{sfig:k5reallow}
\end{subfigure}
\begin{subfigure}[!]{0.32\textwidth}
    \centering
    \includegraphics[width = \textwidth]{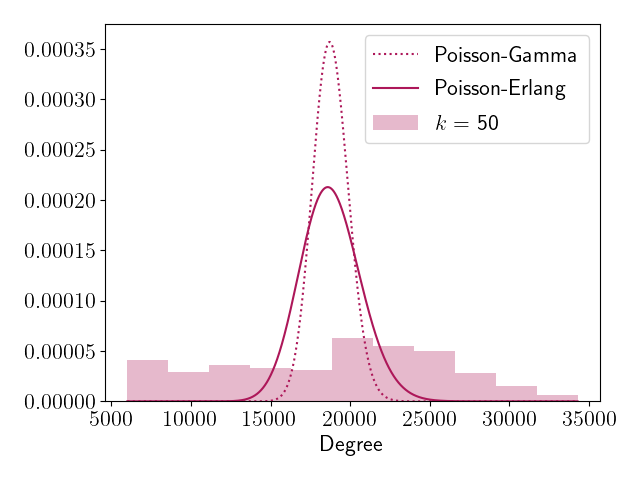}
    \caption{$k = 50$}
    \label{sfig:k50reallow}
\end{subfigure}
    \caption{Degree distribution in Enschede grid for $k = 1, 5$ and $50$ and weak shadowing ($\sigma = 0.1$).}
    \label{fig:limiting_distribution_real}
\end{figure}

\begin{figure}[h!]
\begin{subfigure}[!]{0.32\textwidth}
    \centering
    \includegraphics[width = \textwidth]{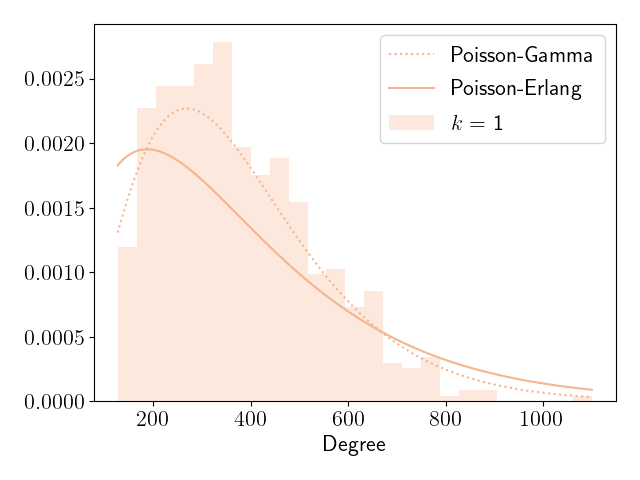}
    \caption{$k = 1$}
    \label{sfig:k1real}
\end{subfigure}
\hfill
\begin{subfigure}[!]{0.32\textwidth}
    \centering
    \includegraphics[width = \textwidth]{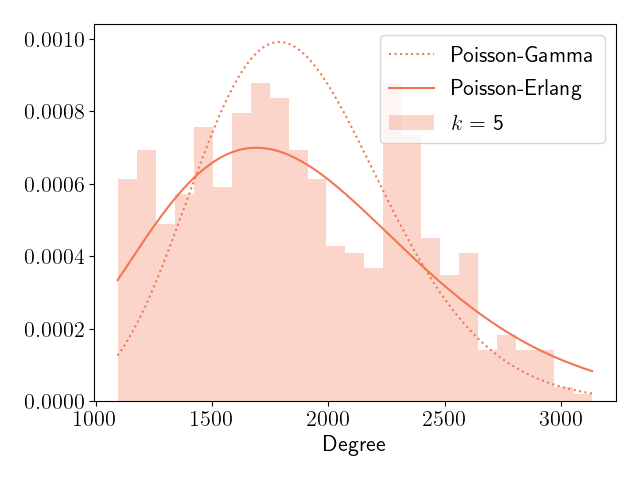}
    \caption{$k = 5$}
    \label{sfig:k5real}
\end{subfigure}
\begin{subfigure}[!]{0.32\textwidth}
    \centering
    \includegraphics[width = \textwidth]{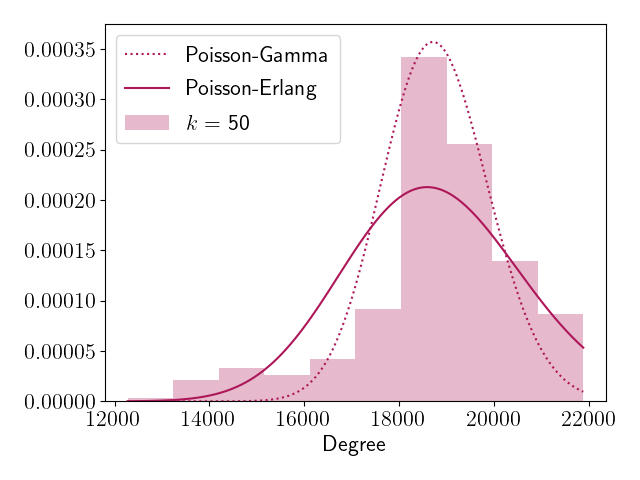}
    \caption{$k = 50$}
    \label{sfig:k50real}
\end{subfigure}
    \caption{Degree distribution in Enschede grid for $k = 1, 5$ and $50$ and strong shadowing ($\sigma = 1$). }
    \label{fig:limiting_distribution_real_sigma1}
\end{figure}
Figure \ref{sfig:k1reallow} shows that for $k=1$ both the compound Poisson-Gamma and the compound Poisson-Erlang distributions seem to fit reasonably well. However, both distributions do not fit well for larger values of $k$ (Figures \ref{sfig:k5reallow}--\ref{sfig:k50reallow}). We can conclude from this that assuming that base stations are distributed by a Poisson point process works well in the case users only connect to 1 base station, but quickly loses accuracy when users are connected to multiple base stations. However, when strong shadowing is taking place, as is the case in Figure \ref{fig:limiting_distribution_real_sigma1}, the compound Poisson degree distributions seem to fit better, not only for $k = 1$. This result implies that our results of the Poisson process can be accurate for a very wide range of spatial processes, when a process of randomness is present as well.

\begin{figure}
    \centering
    \includegraphics[width = 0.5\textwidth]{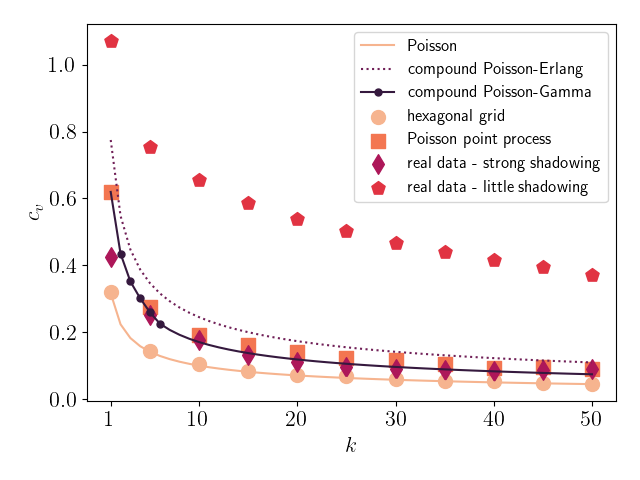}
    \caption{Coefficient of variation for different $k$ and different degree distributions. The markers show the simulated $c_V$'s derived from the simulations in Figures \ref{fig:limiting_distribution_hexagonal}--\ref{fig:limiting_distribution_real_sigma1}, and the lines are Equations \eqref{eq:cv1}, \eqref{eq:cv2} and \eqref{eq:cv3}.}
    \label{fig:cv}
\end{figure}

The coefficients of variation for the degree distribution of the Poisson point process, the hexagonal grid as well as the real data are shown in Figure \ref{fig:cv} for $\lb = \frac{\atot}{N} \approx 0.01$ and $\la = 0.1$. For the compound Poisson-Gamma distribution, we plotted the values of $c_V$ in \eqref{eq:cv3} with linear extrapolated $a_k$ for $k > 5$, using the values of Table \ref{tab:simulated_values_area}. The behaviour of the coefficient of variation for the real grid with shadowing is similar to the one of the Poisson point process, and decreases rapidly in $k$. The compound Poisson-Gamma distribution approximates this coefficient of variation closely, but the compound Poisson-Erlang distribution also works reasonably well as it only slightly overestimates $c_V$. \\

The behaviour of the coefficient of variation for the real grid without shadowing is different from the other three $c_V$'s depicted in Figure \ref{fig:cv} and cannot be approximated by one of the three degree distributions given in \eqref{eq:cv1}, \eqref{eq:cv2} and \eqref{eq:cv3}. Again, the coefficient of variation does slowly decrease, but it is still significant larger than the other three $c_V$'s. This implies that for non-Poissonian data, larger values of $k$ will still result in a more concentrated degree distribution, but not as concentrated as in the Poissonian data. This observation is also shown in Figure \ref{fig:limiting_distribution_real}, as the degree distribution for $k = 50$ does not resemble a concentrated distribution.\\

In general, this plot shows that the coefficient of variation always decreases for larger values of $k$, which means that the load of all connections becomes more evenly balanced among all $B$-points. In the context of wireless networks, this can imply that all $A$-points, the users, will receive a more similar throughput, and the $B$-points, the base stations, will have similar degrees, which results in a more fair distribution of the resources over all users and base stations.

\section{Conclusion}
In this paper, we have derived degree distributions for $AB$ random geometric graphs for different spatial distributions of $B$-points and Poisson-distributed $A$-points. In the case where $B$ points are distributed as a hexagonal grid, we showed that the areas in which a $B$-point is $k$-closest are equal for every point $i\in B$ and every value of $k$ in both one and two dimensions. With this observation, we derived the degree distribution of the $B$-points in one and two dimensions. In the case where $B$ points are distributed as a 1-dimensional Poisson point process, we derived the analytical size distribution of the areas in which a $B$-point is $k$-th closest, which resulted in a compound Poisson-Erlang degree distribution \eqref{eq:1dpoisson} for the degrees of all $B$-points. We fitted a one-parameter Gamma distribution to obtain the $k$-th closest area distribution in the case where B-points are distributed as a 2-dimensional Poisson point process for $k \in \{1, 2, 3, 4, 5\}$. This results in a compound Poisson-Gamma distribution for the degree distribution of $B$-points. For larger $k$, we show that the easier compound Poisson-Erlang degree distribution works well as an approximation for the degree distribution of $B$-points. \\

Moreover, we have shown that the coefficient of variation of the degree distributions for both the hexagonal grid model and the Poisson point process rapidly decrease for larger values of $k$. Therefore the degrees become more centered around the mean as $k$ increases. This can have important implications for applications of $AB$ random graphs. For example, for multi-connected cellular networks, this means that for large $k$, the load in the network becomes more evenly distributed (\textit{fairness}). Investigating the extent to which fairness increases with increasing $k$ is therefore an interesting topic for further research.\\ 

In a case study with real data of base station locations, we have shown that with strong shadowing, which introduces a source of randomness in the observed distance (Figure \ref{fig:limiting_distribution_real_sigma1}), our derived degree distributions for the 1-dimensional and the 2-dimensional Poisson point process approximate the real degree distribution well, even though these data are not distributed according to a Poisson point process. This is in line with \cite{blaszczyszyn2012quality, keeler2018wireless}, in which the authors found that wireless networks appear to be Poisson under strong shadowing. Moreover, it can also be seen that for almost no shadowing (Figure \ref{fig:limiting_distribution_real}), the degree distribution in the data behaves significantly different from the Poisson case, especially for larger values of $k$. A reason for this could be that real base stations are not independently distributed among the grid, which is a key property of the Poisson point process. This mismatch becomes more visible for larger degrees of multi-connectivity, as in this case more base stations and thus more dependencies need to be taken into account. This means that especially when one wants to investigate multi-connectivity in a network with little to no shadowing, it is important to investigate whether the Poisson point process is a suitable model for distributing the $B$-points. Even if it seems to fit well for $k=1$, the fit for larger values of $k$ may be significantly worse, comparing Figures \ref{sfig:k5reallow}--\ref{sfig:k50reallow} with Figures \ref{sfig:k5poisson}--\ref{sfig:k50poisson}. \\

In this research, we have assumed that $A$-points are always distributed as a Poisson point process. However, the locations of the $A$-points can also depend on the locations of the $B$-points in many application areas. For example, $A$ points may follow a heterogeneous distribution instead of a homogeneous Poisson distribution, that depends on dense parts and less dense parts of $B$-points in the spatial process. Deriving degree distributions and results on load balancing for those types of $AB$-random graphs is an interesting topic of further research.

\paragraph{Acknowledgements} The authors would like to thank Nelly Litvak and Suzan Bayhan for useful discussions on the manuscript.

\bibliographystyle{siam}
\bibliography{biblio}

\end{document}